\documentclass[preprint,12pt]{elsarticle}

\usepackage{mathrsfs}
\usepackage{amssymb}
\usepackage{amsthm}
\usepackage{amsmath}
\usepackage{graphicx}
\usepackage{subfig}

\newtheorem{theorem}{Theorem}[section]
\newtheorem{lemma}[theorem]{Lemma}
\newtheorem{conj}[theorem]{Conjecture}
\newtheorem{coro}[theorem]{Corollary}
\newtheorem{prop}[theorem]{Proposition}

\allowdisplaybreaks

\begin{document}

\begin{frontmatter}

\title{More on the full Brouwer's Laplacian \\ spectrum conjecture}

\author{Xiaodan Chen}%
\ead{x.d.chen@live.cn}

\author{Junwei Zi}
\ead{junwzmath@163.com}


\address{College of Mathematics and Information Science, Guangxi University,\\
Nanning 530004, Guangxi, P. R. China}

\begin{abstract}
Brouwer conjectured that the sum of the first $k$ largest Laplacian eigenvalues of an $n$-vertex graph
is less than or equal to the number of its edges plus $\binom{k+1}{2}$ for each $k\in \{1,2,\cdots,n\}$,
which has come to be known as Brouwer's conjecture.
Recently, Li and Guo further considered the case when the equalities hold in these conjectured inequalities,
and proposed the full version of Brouwer's conjecture.
In this paper, we first present a concise version of the full Brouwer's conjecture.
Then we show that the full Brouwer's conjecture holds for
two families of spanning subgraphs of complete split graphs
and for $c$-cyclic graphs with $c\in\{0,1,2\}$.
We also consider the Nordhaus-Gaddum version of the full Brouwer's conjecture and present partial solutions to it.

\end{abstract}


\end{frontmatter}

\section{Introduction}

It is well-known that the Laplacian spectrum of a graph encodes abundant information about combinatorial properties of the graph.
One of the famous examples is Kirchhoff's matrix-tree theorem \cite{Kirchhoff},
which tells us that the number of spanning trees in a connected graph is equal to the product of all non-zero Laplacian eigenvalues of the graph divided by its order.
Another example is Grone-Merris conjecture posed in 1994 \cite{Grone1},
which states that the (non-increasing) Laplacian eigenvalue sequence of a graph is majorized by the conjugate degree sequence of the graph
(see Section 2 below for details).
This conjecture has been solved by Bai \cite{Bai}, and now is known as Grone-Merris-Bai theorem.

\begin{theorem}[Grone-Merris-Bai theorem]\label{thm-Bai} %
 	For any graph $G$ on $n$ vertices with (non-increasing) Laplacian eigenvalue sequence $\mu(G):=(\mu_1,\mu_2,\cdots,\mu_n)$
    and conjugated degree sequence $d^*(G):=(d_1^*,d_2^*,\cdots,d_n^*)$, $\mu(G)$ is majorized by $d^*(G)$, namely,
 	$$\mu(G) \preceq d^*(G).$$
    Moreover, the equality holds if and only if $G$ is a threshold graph.
\end{theorem}

We should mention that the above characterization for the case of equality is due to Merris \cite{Merris2},
which also follows from a more general result of Duval and Reiner; see Proposition 6.4 in \cite{Duval}.

On the other hand, as a variant of Grone-Merris conjecture, 
Brouwer \cite{Brouwer1} proposed the following conjecture,
which has come to be known as Brouwer's conjecture.

\begin{conj} [Brouwer's conjecture] \label{conj-1}
For any graph $G$ on $n$ vertices with $m$ edges and for each $k\in\{1,2,\cdots,n\}$,
\begin{eqnarray*}
s_{k}(G):=\sum_{i=1}^{k}\mu_i\leq  m+\binom{k+1}{2}. 
\end{eqnarray*}
\end{conj}
By virtue of computers, Brouwer's conjecture was confirmed for graphs with at most 11 vertices \cite{Brouwer1,Cooper}.
It was also proved that Brouwer's conjecture is true for trees \cite{Haemers1}, unicyclic graphs \cite{Du,Wang}, bicyclic graphs \cite{Du},
threshold graphs \cite{Brouwer1}, split graphs \cite{Mayank}, cographs \cite{Mayank}, and regular graphs \cite{Mayank},
and is true for $k\in\{1,2,n-3,n-2,n-1,n\}$ \cite{Haemers1,Chen3}.
Helmberg and Trevisan \cite{Helmberg} showed that
a graph satisfies Brouwer's conjecture if and only if it is spectrally threshold dominated,
which presents a combinatorial condition equivalent to Brouwer's conjecture.
Rocha \cite{Rocha2} proved that Brouwer's conjecture holds for a sequence of random graphs with probability tending to one as the number of vertices goes to infinity.
For other progress on Brouwer's conjecture we refer to \cite{Blinovsky,Chen1,Chen2,Cooper,Ferreira,Ganie1,Ganie2,Ganie3,Ganie4,Ganie5,Rocha1,Torres}.
It should be noted, however, that Brouwer's conjecture is still open so far.

Recently, Li and Guo \cite{Li} further considered the case where the equalities hold in the inequalities of Brouwer's conjecture,
and proposed the full version of Brouwer's conjecture.

\begin{conj} [The full Brouwer's conjecture] \label{conj-2}
	For any graph $G$ on $n$ vertices with $m$ edges and for each $k\in\{1,2,\cdots,n-1\}$,
	\begin{eqnarray*}
		s_{k}(G)\leqslant m+\binom{k+1}{2},
	\end{eqnarray*}
    with equality if and only if $G \cong G_{k,r,s}$ ($r \geqslant 1$, $s \geqslant 0$),
    where $G_{k,r,s}$ is the graph of order $n=k+r+s$ consisting of a clique $K_{k}$ and two independent sets $\overline{K_{r}}$ and $\overline{K_{s}}$,
    such that each vertex in $K_{k}$ is adjacent to all vertices in $\overline{K_{r}}$,
    and for each vertex $v_{i}$ in $\overline{K_{s}}$, $N(v_{i}) \subsetneq  V(K_{k})$ ($i=1,2,\cdots,s$), and $N(v_{i+1}) \subseteq N(v_{i})$ ($i=1,2,\cdots,s-1)$.
\end{conj}

Observe that $G_{k,r,s}$ is a threshold graph having $n=k+r+s$ vertices and clique number $k+1$, as shown in the proof of Theorem 2.2 in \cite{Li}.
Here we remark that any threshold graph having $n$ vertices and clique number $k+1$ must be the graph $G_{k,r,s}$ with appropriate $r$ and $s$.
To see this, let us recall that a graph is threshold if and only if it can be constructed through an iterative process which starts with an isolated vertex,
and where, at each step, either a new isolated vertex is added, or a new dominating vertex (i.e., a vertex adjacent to all previous vertices) is added \cite{Mahadev}.
This will yield that every threshold graph corresponds one-to-one to a $\{0,1\}$-sequence,
where $0$ and $1$ record the operations of adding an isolated vertex and a dominating vertex, respectively.
In particular, the corresponding sequence of a threshold graph having $n$ vertices and clique number $k+1$ is a sequence of length $n$ with exactly $k$ $1$'s,
which has the following general form (where $1<i_1<i_2<\cdots<i_k\leq n$):
$$\frac{0}{\mathop{1}\limits^\uparrow}\cdots \frac{1}{\mathop{i_1}\limits^\uparrow}\cdots\frac{1}{\mathop{i_2}\limits^\uparrow}\cdots\cdots \frac{1}{\mathop{i_k}\limits^\uparrow}\cdots$$
In this case, one can easily see that the $k$ dominating vertices form the clique $K_{k}$,
and the first $(i_1-1)$ isolated vertices form the independent set $\overline{K_{r}}$ (let $r=i_1-1$),
and the remaining isolated vertices form the independent set $\overline{K_{s}}$
(let $s=n-k-i_1+1$ and let $v_1,v_2,\cdots,v_s$ be the corresponding isolated vertices from left to right along the sequence);
this is exactly the graph $G_{k,r,s}$. 

Now, we can give a concise version of Conjecture \ref{conj-2}.

\begin{conj} [A concise version of the full Brouwer's conjecture] \label{conj-3}
	For any graph $G$ on $n$ vertices with $m$ edges and for each $k\in\{1,2,\cdots,n-1\}$,
	\begin{eqnarray*}
		s_{k}(G)\leqslant m+\binom{k+1}{2},
	\end{eqnarray*}
	with equality if and only if $G$ is a threshold graph having $n$ vertices and clique number $k+1$.
\end{conj}

By computers, Li and Guo \cite{Li} confirmed the full Brouwer's conjecture for graphs with at most 9 vertices.
Moreover, they proved that the conjecture is true for $k\in\{1,2,n-3,n-2$,$n-1$\},
where the case of $k=2$ gives a solution to a conjecture of Guan et al. \cite{Guan};
see \cite{Zheng1,Zheng2} for more related results on this aspect.
In this paper, we proceed to explore the full Brouwer's conjecture.

Nordhaus-Gaddum-type results for various graph parameters have been extensively studied; see \cite{Aouchiche} for a comprehensive survey.
Here we would like to investigate Nordhaus-Gaddum-type results for $s_{k}(G)$.
Inspired by the (full) Brouwer's conjecture, we pose the following conjecture.

\begin{conj} \label{conj-4}
	Let $G$ be a graph with $n$ vertices and $\overline{G}$ be its complement. Then for each $k\in\{1,2,\cdots,n-1\}$,
	\begin{eqnarray*}
		s_k(G)+s_k(\overline{G})\leq \binom{n}{2}+2\cdot\binom{k+1}{2},
	\end{eqnarray*}
    with equality if and only if $G$ is a threshold graph having $n=2k+1$ vertices and clique number $k+1$.
\end{conj}

Notice that the full Brouwer's conjecture implies Conjecture \ref{conj-4} and thus,
the (sufficient) conditions for which the full Brouwer's conjecture holds also apply to Conjecture \ref{conj-4};
in particular, Conjecture \ref{conj-4} holds for $k\in\{1,2,n-3,n-2,n-1\}$.
Conversely, if Conjecture \ref{conj-4} is true, then the full Brouwer's conjecture is also true for self-complement graphs
(i.e., the graphs that are isomorphic to their complements).

The rest of the paper is organized as follows.
In Section 2, we will present a brief introduction to necessary terminology and notation,
and give some lemmas that will be used to prove our main results in the coming sections.
In Section 3, we shall prove that the full Brouwer's conjecture is true for
two families of spanning subgraphs of complete split graphs
and for $c$-cyclic graphs with $c\in\{0,1,2\}$,
and present partial solutions to Conjecture \ref{conj-4}.
Finally, a concluding remark will be made in Section 4.

\section{Preliminaries}

We only consider finite and undirected graphs without multiple edges and self-loops.
Given a graph $G$, we write $V(G)$ for its vertex set and $E(G)$ for its edge set, and let $|V(G)|=n$ and let $e(G):=|E(G)|$.
If $H$ is a subgraph of $G$, let $G-E(H)$ denote the graph obtained by removing the edges in $H$ from $G$.
The complement of $G$ is defined to be $\overline{G}:=K_n-E(G)$,
where $K_n$ is the complete graph with $n$ vertices.
For two vertex-disjoint graphs $G_1$ and $G_2$, the union of $G_1$ and $G_2$, denoted by $G_1\cup G_2$,
is the graph having the vertex set $V(G_1)\cup V(G_2)$ and the edge set $E(G_1)\cup E(G_2)$,
whereas the join of $G_1$ and $G_2$, denoted by $G_1\vee G_2$, is the graph having the vertex set $G_1\cup G_2$
and the edge set $E(G_1)\cup E(G_2)\cup\{uv: u\in V(G_1), v\in V(G_2)\}$.
We also denote by $tG$ the vertex-disjoint union of $t$ copies of $G$.
An independent set of a graph $G$ is a set of mutually non-adjacent vertices in $G$,
whereas a clique of $G$ is a set of mutually adjacent vertices in $G$;
the maximum size of a clique of $G$, denoted by $\omega(G)$, is known as the clique number of $G$.

The degree sequence of a graph $G$ is defined to be
$d(G):=(d_1,d_2,\cdots,d_n),$
where $d_i$ is the degree of a vertex in $G$ and assume, without loss of generality, that $d_1\geq d_2\geq \cdots\geq d_n$.
Observe that $d(G)$ is an $n$-partition of $2e(G)$.
The conjugate of $d(G)$, known as the conjugate degree sequence of the graph $G$, is the sequence
$d^*(G):=(d_1^*,d_2^*,\cdots,d_n^*)$, where $d_i^*:=|\{j: d_j\geq i\}|$.
It is easy to see that $d_1^*\geq d_2^*\geq \cdots\geq d_n^*=0$.
The conjugate degree sequence of $G$ is conveniently visualized by means of Ferrers-Sylvester (or Young) diagram,
which consists of $n$ left-justified rows of boxes with $d_i$ boxes in row $i$.
In this diagram, the conjugate degree $d_i^*$ counts the number of boxes in column $i$.
An interesting fact concerning $d(G)$ and $d^*(G)$ is that \cite{Ruch}
\begin{eqnarray}
d_i+1\leq d_i^*\,\, \mbox{for}\,\, 1\leq i\leq T, \label{eq-1-1}
\end{eqnarray}
where $T:=\max\{i: d_i\geq i\}$ is known to be the trace of $d(G)$. 
Moreover, the equalities hold in (\ref{eq-1-1}) for $1\leq i\leq T$ if and only if $G$ is a threshold graph \cite{Merris2}.
It is also shown in \cite{Merris2} that if $G$ is a threshold graph, then
\begin{eqnarray*}
d_{i+1}=d_i^*\,\, \mbox{for}\,\, T+1\leq i\leq n-1. 
\end{eqnarray*}
These yield that in the Ferrers-Sylvester diagram of any threshold graph
the part on and above the diagonal boxes is exactly the transpose of the part below the diagonal. 
See Figure 1 for an illustration, 
where the left hand side is the Ferrers-Sylvester diagram of the (threshold) graph given by the degree sequence $d(G)=(7, 5, 4, 3, 3, 2, 1, 1)$,
whereas the right hand side draws the general appearance of the Ferrers-Sylvester diagram of a threshold graph.

\begin{figure}
  \centering
  \includegraphics[width=12cm]{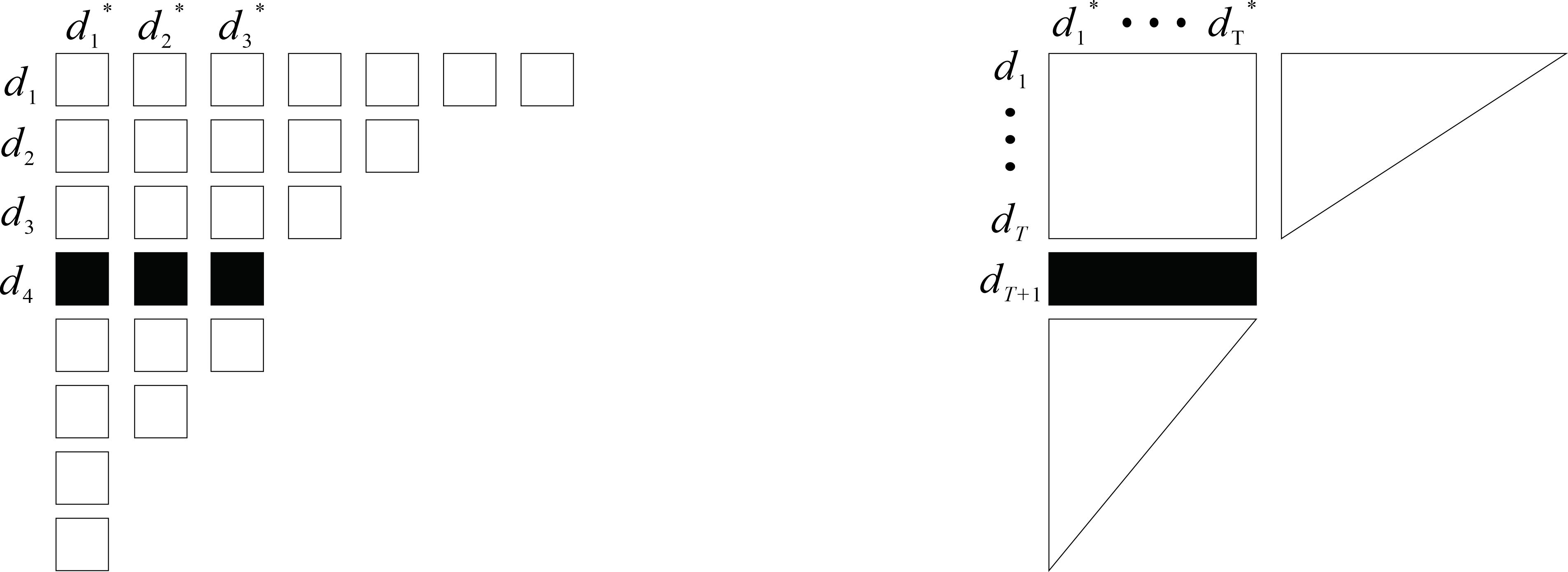}\\
  \caption{The Ferrers-Sylvester diagram of a threshold graph.} \label{fig1}
\end{figure}

There is also a constructive characterization for threshold graphs, that is,
a graph is threshold if and only if it can be constructed through an iterative process which starts with an isolated vertex,
and where, at each step, either a new isolated vertex is added, or a new dominating vertex
(i.e., a vertex adjacent to all previous vertices) is added \cite{Mahadev}.
Obviously, given a threshold graph $G$, all dominating vertices form a clique of $G$,
whereas the remaining isolated vertices form an independent set of $G$;
such a graph is also called a split graph, that is, the graph whose vertex set can be partitioned into a clique $V_1$ and an independent set $V_2$.
In particular, if every vertex in $V_2$ is adjacent to every vertex in $V_1$, then it is known as a complete split graph, which is also a threshold graph.
Note that, in general, a split graph is not necessarily threshold.

A $c$-cyclic graph is a connected graph in which the number of edges equals the number of vertices plus $c-1$.
Customarily, we refer to a $c$-cyclic graph as a tree, unicyclic graph, and bicyclic graph when $c=0,1,$ and $2$, respectively.
Intuitively speaking, a tree is a connected graph without cycles, whereas a unicyclic graph is a cycle with possible trees attached.
For a bicyclic graph $G$, its base, denoted by $\widehat{G}$, is the (unique) minimal bicyclic subgraph of $G$.
Note that a bicyclic graph can always be obtained by attaching trees to some vertices of its base.
It is known that \cite {GuoS} there are just two types of bases for bicyclic graphs, that is,
the $\infty$-graph $\infty(p,l,q)$ ($p\geq q\geq 3, l\geq 1$) and $\theta$-graph $\theta(p,l,q)$ ($p\geq q\geq l\geq 1$), which are pictured in Figure 2.

\begin{figure}
  \centering
  \includegraphics[width=12cm]{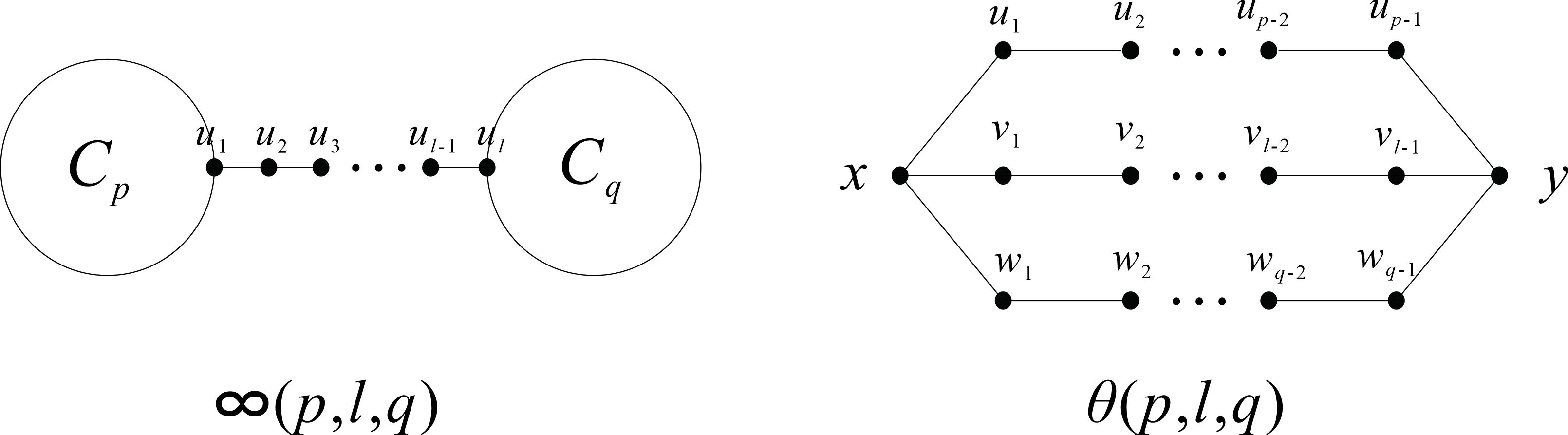}\\
  \caption{The $\infty$-graph $\infty(p,l,q)$ and $\theta$-graph $\theta(p,l,q)$.} \label{fig2}
\end{figure}

The Laplacian matrix of a graph $G$ is defined to be $L(G)=D(G)-A(G)$,
where $A(G)$ is the adjacency matrix of $G$ and $D(G)$ is the diagonal matrix of vertex degrees of $G$.
Notice that $L(G)$ is positive-semidefinite, and hence has non-negative real numbers as eigenvalues,
which are usually denoted in non-increasing order by $\mu_1\geq\mu_2\geq\cdots\geq\mu_n$.
When more than one graph is under discussion, we write $\mu_i(G)$ instead of $\mu_i$.
We also let $\mu(G):=(\mu_1,\mu_2,\cdots, \mu_n)$, and call it the Laplacian eigenvalue sequence of the graph $G$.
It is well known that $\mu_n=0$ and, $\mu_{n-1}>0$ if and only if $G$ is connected (see, e.g., \cite{Brouwer1}).

The Laplacian eigenvalue sequence $\mu(G)$ of a graph $G$ is related intimately to its conjugated degree sequence $d^*(G)$,
just as shown by Grone-Merris-Bai theorem,
$$\mu(G) \preceq d^*(G),$$
with equality holding if and only if $G$ is a threshold graph.
Here $\preceq$ indicates the majorization partial order, that is,
$$\sum_{i=1}^{k} \mu_{i} \leq \sum_{i=1}^{k} d_{i}^*\,\, \text{for}\,\, 1\leq k<n,\,\, \text{and}\,\, \sum_{i=1}^{n} \mu_{i}=\sum_{i=1}^{n} d_{i}^*.$$
In this case, we say that $\mu(G)$ is majorized by $d^*(G)$.

There are some other properties of the Laplacian eigenvalues of a graph that we shall use frequently in the following sections.

\begin{lemma}[see \cite{Merris}] \label{lm-1}
	(i) Let $G$ be a graph with $n$ vertices and let $\overline{G}$ be its complement. Then
	$\mu_n(\overline{G})=0$ and $\mu_i(\overline{G})=n-\mu_{n-i}(G)$ for $i=1,2,\cdots,n-1$.

    (ii) For any graph $G$ with $n$ vertices, we have $\mu_1(G)\leq n$.

    (iii) Let $G_1$ and $G_2$ be two vertex-disjoint graphs with $n_1$ and $n_2$ vertices, respectively.
	Then the Laplacian eigenvalues of $G_1\vee G_2$ are $n_1+n_2, n_1+\mu_1(G_2),\cdots,n_1+\mu_{n_2-1}(G_2), n_2+\mu_1(G_1),\cdots,n_2+\mu_{n_1-1}(G_1)$, and $0$.

    (iv) Let $G$ be a graph with $n$ vertices and let $G^\prime$ be the graph obtained by deleting an arbitrary edge from $G$.
    Then the Laplacian eigenvalues of $G$ and $G^\prime$ interlace, that is,
    $\mu_1(G)\geq \mu_1(G^\prime)\geq \mu_2(G)\geq\mu_2(G^\prime)\geq \cdots\geq\mu_{n-1}(G)\geq\mu_{n-1}(G^\prime)\geq \mu_n(G)=\mu_n(G^\prime)=0$.
\end{lemma}

\begin{lemma}[see \cite{Du}] \label{lm-2}
Let $T_{n}^i$ be the tree of order $n$ as shown in Figure \ref{fig3}.
For $n\geq 6$ (resp. $n\geq 7$) and $i=2$ (resp. $i=3$), we have $1<\mu_2(T_{n}^i)<2.7$.
\end{lemma}

\begin{figure}
  \centering
  \includegraphics[width=7cm]{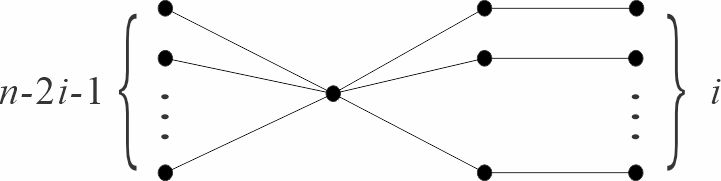}\\
  \caption{The tree $T_n^{i}$, $0\leq i\leq (n-1)/2$.}\label{fig3}
\end{figure}

Let $\lambda_i(M)$ denote the $i$-th largest eigenvalue of a real symmetric matrix $M$ of order $n$.
The next result from matrix theory is known as Fan's inequality.

\begin{lemma}[Fan's inequality \cite{Fan}] \label{lm-3}
	If $B$ and $C$ are real symmetric matrices of order $n$, then for any integer $k$ with $1\leq k\leq n$,
	$$\sum_{i=1}^k\lambda_i(B+C)\leq\sum_{i=1}^k\lambda_i(B)+\sum_{i=1}^k\lambda_i(C).$$
\end{lemma}



For an integer $k$ with $1\leq k\leq n$, let $s_{k}(G)$ denote the sum of the $k$ largest Laplacian eigenvalues of a graph $G$,
i.e., $s_{k}(G)=\sum_{i=1}^{k}\mu_i(G)$.

\begin{lemma}[see \cite{Wang}] \label{lm-5}
    Let $G$ be a connected graph on $n$ vertices. Then for any integer $k$ with $1\leq k\leq n$,
    $$s_{k}(G) \leq 2e(G)-n+2k-\frac{2k-2}{n}.$$
    Moreover, equality is achieved only when $k=1$ and $G\cong S_n$ (i.e., the star on $n$ vertices).\footnote{The characterization for equality follows from Theorem 1.1 in \cite{Frischer}, which is missed in the original version of this result.}
\end{lemma}

\begin{lemma}[see \cite{Zhou}] \label{lm-6}
	Let $G$ be a graph on $n$ vertices.\footnote{From the proof of Theorem 1 in \cite{Zhou},
    it is easy to see that the connectivity condition can be removed when deducing the inequality.}
    Then for any integer $k$ with $1\leq k\leq n-2$,
	$$s_k(G)\leq \frac{2ke(G)+\sqrt{k(n-k-1)\big[n(n-1)-2e(G)\big]e(G)}}{n-1}.$$
    Moreover, if $G$ is connected, then equality holds if and only if $G\cong S_n$ or $K_n$ when $k=1$,
    and $G\cong K_n$ when $2\leq k\leq n-2$.
\end{lemma}

As mentioned in the introduction, the full Brouwer's conjecture has been proven to be true for $k\leq 2$ and $k=n-1$.
When restricting to connected graphs, we have the following.

\begin{lemma}[see \cite{Li}] \label{lm-7}
    Let $G$ be a connected graph on $n$ vertices. Then

    (i) $s_{1}(G) \leq e(G)+1$ with equality if and only if $G\cong S_n$.

    (ii) $s_{2}(G) \leq e(G)+3$ with equality if and only if $G\cong G(s,n-2-s)$, where $0\leq s\leq n-3$ (see Figure \ref{fig4}).

    (iii) $s_{n-1}(G)\leq e(G)+\binom{n}{2}$ with equality if and only if $G\cong K_n$.
\end{lemma}

\begin{figure}
  \centering
  \includegraphics[width=3.8cm]{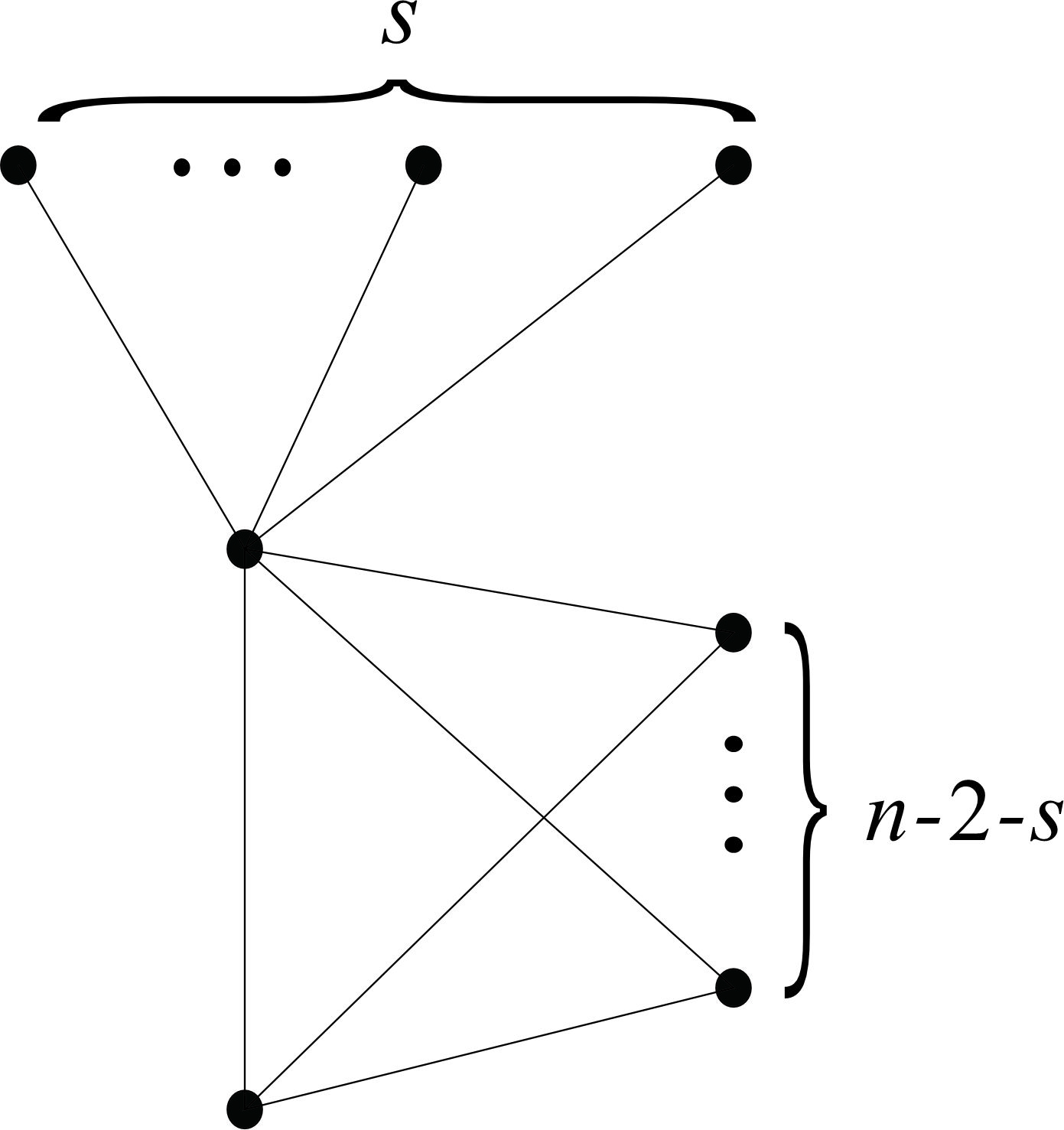}\\
  \caption{The graph $G(s,n-2-s), 0\leq s\leq n-3$.} \label{fig4}
\end{figure}

We also need a Nordhaus-Gaddum-type result for the sum of adjacency eigenvalues, which is due to Nikiforov \cite{Nikiforov}.

\begin{lemma}[see \cite{Nikiforov}] \label{lm-8}
	If $G$ is a graph on $n$ vertices with adjacency matrix $A(G)$, 
	then for $1\leq k\leq (n-1)/2$,
	$$\sum_{i=1}^k\big(|\lambda_{n-i+1}(A(G))|+|\lambda_{n-i+1}(A(\overline{G}))|\big)\leq \sqrt{2k}\bigg(\frac{n}{2}+k\bigg).$$
\end{lemma}

\section{Main results}

In this section, we present more evidence for the full Brouwer's conjecture by proving that several classes of graphs satisfy it,
and give partial solutions to a Nordhaus-Gaddum version of Brouwer's conjecture (i.e., Conjecture \ref{conj-4}).

\subsection{Spanning subgraphs of complete split graphs}

For a start, we consider complete split graphs.
Suppose $G$ is a complete split graph on $n$ vertices with clique number $k+1$.
Then, from the definition of complete split graphs it follows that
$G$ is exactly the join of its vertex-induced subgraphs $G[V_1]$ and $G[V_2]$, that is, $G=G[V_1]\vee G[V_2]$,
where $V_1$ is the clique of size $k$ and $V_2$ is the independent set of size $n-k$.
This yields that $e(G)=\binom{k}{2}+k(n-k)$. Moreover,
by Lemma \ref{lm-1} (iii) and the fact that $\mu(K_k)=(k,\cdots,k,0)$, we obtain
$$\mu(G)=(\underbrace{n,\cdots,n}_k,\underbrace{k,\cdots,k}_{n-k-1},0).$$
Now, a simple calculation shows that
\begin{eqnarray*}
&&s_k(G)=kn=e(G)+\binom{k+1}{2},\\
&&s_t(G)=tn<e(G)+\binom{t+1}{2}\,\, \textrm{for}\,\, 1\leq t<k,\\
&&s_t(G)=kn+(t-k)k<e(G)+\binom{t+1}{2}\,\, \textrm{for}\,\, k<t\leq n-1,
\end{eqnarray*}
from which we can conclude the following.

\begin{prop}\label{prop-1}	
The full Brouwer's conjecture is true for complete split graphs.
\end{prop}

As we see, complete split graphs are so restrictive.
To take a step forward, one might naturally ask whether the full Brouwer's conjecture holds for any spanning subgraph of complete split graphs.
We here do not present a complete answer to this question,
but find two kinds of such spanning subgraphs satisfying the full Brouwer's conjecture:
one is the split graph, which can be obtained from a complete split graph by removing some edges between the clique and independent set;
the other is the graph obtained from a complete split graph by removing some edges whose endpoints are both in the clique.

\begin{theorem}\label{thm-split}	
The full Brouwer's conjecture is true for split graphs.
\end{theorem}

\begin{proof}
Let $G$ be a split graph on $n$ vertices.
For $1\leq k\leq n-1$, we consider the following function (in variable $k$)
$$f_{k}(G):=e(G)+\binom{k+1}{2}-\sum_{i=1}^{k} d_{i}^{*}.$$
It was shown in \cite{Mayank} that $f_{k}(G)\geq 0$ with equality if $k=T$ (i.e., the trace of the degree sequence of $G$).
This means that the Grone-Merris's bound on $s_k(G)$ is better than the Brouwer's bound when $G$ is a split graph, that is,
\begin{eqnarray}
\sum_{i=1}^{k} d_{i}^{*}\leq e(G)+\binom{k+1}{2},\, 1\leq k\leq n-1. \label{eq-3-1-1}
\end{eqnarray}

If $G$ is not a threshold graph, then combining Grone-Merris-Bai theorem and (\ref{eq-3-1-1}), we have
$$s_k(G)<\sum_{i=1}^{k} d_{i}^{*}\leq e(G)+\binom{k+1}{2},\, 1\leq k\leq n-1,\,\, \textrm{as desired.}$$

If $G$ is a threshold graph, from the constructive characterization of threshold graphs we see that
all dominating vertices plus the initial isolated vertex form a maximum clique of $G$,
which implies that the number $N$ of dominating vertices in $G$ is exactly the clique number of $G$ minus 1, i.e., $N=\omega(G)-1$.
On the other hand, observing that every dominating vertex of $G$ has degree at least $N$,
whereas every isolated vertex of $G$ has degree at most $N$ with the degree of the initial isolated vertex being exactly $N$,
we get $$d_1\geq\cdots\geq d_N\geq N=d_{N+1}\geq\cdots \geq d_n,$$ and hence, $T=N$.
Consequently, we have $T=\omega(G)-1$.
We next consider the following three cases: 

\textbf{Case 1.} $k=T=\omega(G)-1$.

In this case, $G$ has clique number $k+1$.
Let us now consider the Ferrers-Sylvester diagram of $G$, which, as mentioned in previous section,
has a certain symmetry, that is,
the part on and above the diagonal boxes is exactly the transpose of the part below the diagonal.
This would allow us to partition the diagram into three parts: $X$, $Y$ and $Z$; see Figure \ref{fig5} for an illustration.
Denote by $n_X$, $n_Y$ and $n_Z$ the number of boxes in Parts $X$, $Y$ and $Z$, respectively.
By symmetry we get
\begin{eqnarray*}
n_Y=n_Z\,\, \text{and}\,\, n_X=T(T+1).
\end{eqnarray*}
This, as well as the fact that $2e(G)=\sum_{i=1}^n d_i=n_X+n_Y+n_Z$, yields that
\begin{eqnarray*}
n_Y=n_Z=e(G)-\frac{T(T+1)}{2}.
\end{eqnarray*}
Consequently, from Grone-Merris-Bai theorem it follows directly that
$$s_k(G)=s_T(G)=\sum_{i=1}^T d_i^*=n_X+n_Y=e(G)+\frac{T(T+1)}{2}=e(G)+\binom{k+1}{2}.$$

\begin{figure}
  \centering
  \includegraphics[width=4.6cm]{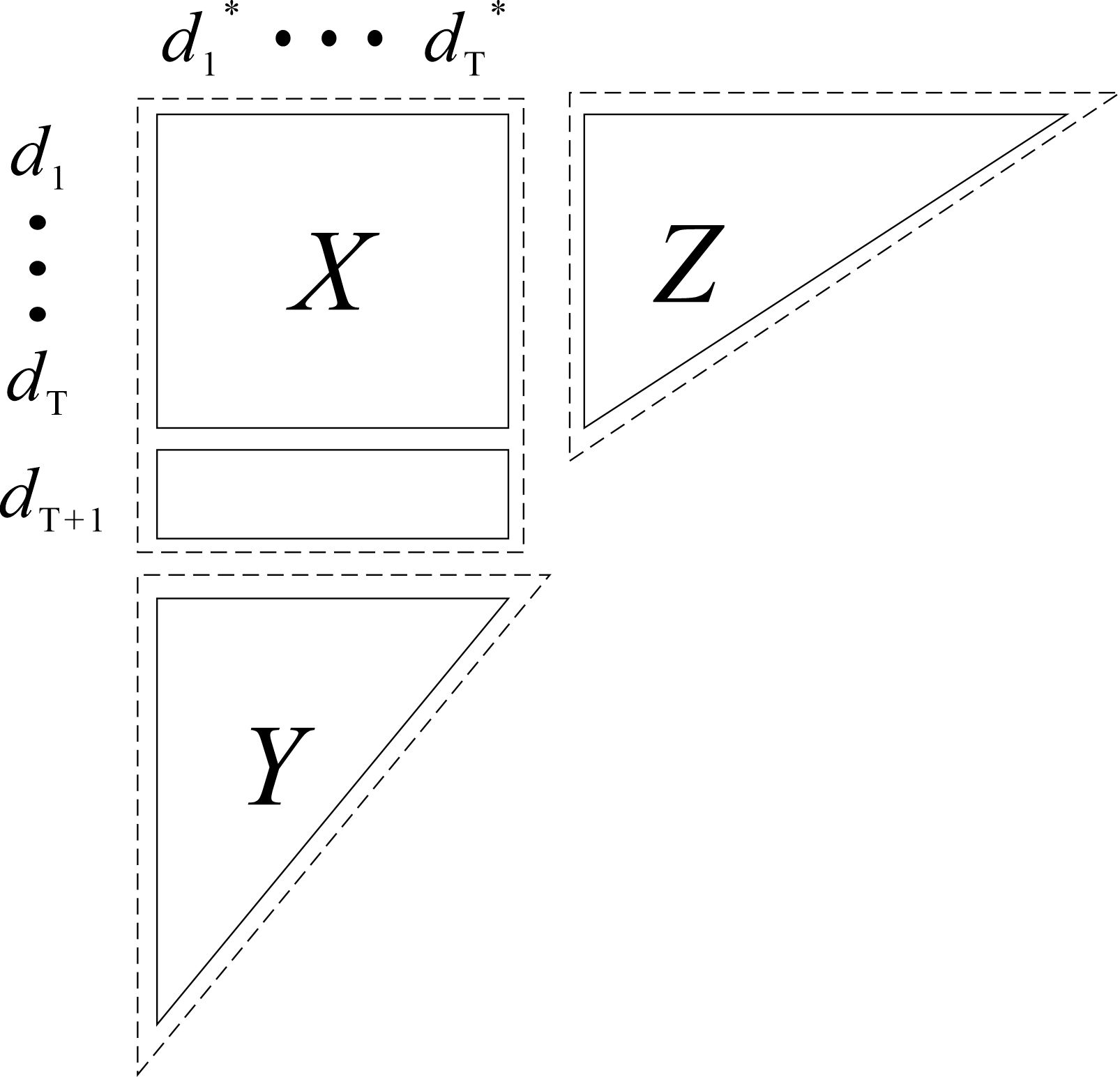}
  \caption{Parts $X$, $Y$ and $Z$ in the Ferrers-Sylvester diagram of a threshold graph $G$.}\label{fig5}
\end{figure}

\textbf{Case 2.} $1\leq k\leq T-1=\omega(G)-2$.

From the Ferrers-Sylvester diagram of $G$ we see that $d_T^*\geq T+1$.
Thus, again by Grone-Merris-Bai theorem, we have
\begin{eqnarray*}
s_k(G)=\sum_{i=1}^k d_i^*&=&\sum_{i=1}^T d_i^*-\sum_{i=k+1}^T d_i^*\\
&\leq& e(G)+\frac{T(T+1)}{2}-(T-k)(T+1) \\
&=& e(G)+\binom{k+1}{2}-\binom{T-k+1}{2} \\
&<& e(G)+\binom{k+1}{2}.
\end{eqnarray*}

\textbf{Case 3.} $\omega(G)=T+1\leq k\leq n-1$.

As Case 2, by the Ferrers-Sylvester diagram of $G$ we have $d_{T+1}^*\leq T$ and,
\begin{eqnarray*}
s_k(G)=\sum_{i=1}^k d_i^*&=&\sum_{i=1}^T d_i^*+\sum_{i=T+1}^k d_i^*\\
&\leq& e(G)+\frac{T(T+1)}{2}+(k-T)T \\
&=& e(G)+\binom{k+1}{2}-\binom{k-T+1}{2} \\
&<& e(G)+\binom{k+1}{2}.
\end{eqnarray*}

Now, by combining the above discussion, we can conclude that for a split graph $G$ on $n$ vertices and for $1\leq k\leq n-1$,
$$s_k(G)\leq e(G)+\binom{k+1}{2},$$
with equality if and only if $G$ is a threshold graph on $n$ vertices with clique number $k+1$.
This completes the proof.
\end{proof}

We remark that the assertion proved in Case 1 has appeared in different but equivalent form in other places; see e.g., \cite{Abebe,Das,Li}.
Here, for completeness, we present a proof of it (as well as Case 2 and Case 3),
where the idea of our proof comes from \cite{Das}; see Corollary \ref{thm-4-3} below for an alternative proof.

\begin{theorem}\label{thm-3-1-1}
	Let $G$ be an arbitrary graph on $p$ vertices.
	If $p\leq q$, then for $1\leq k\leq p+q-1$,
	$$s_{k}(G\vee qK_1)\leq e(G\vee qK_1)+\binom{k+1}{2},$$
	with equality if and only if $k=p$ and $G\cong K_{p}$ (in this case, $G\vee qK_1$ is a threshold graph with clique number $k+1$).
\end{theorem}

\begin{proof}
	Observe first that
	\begin{eqnarray}
		e(G\vee qK_1)=pq+e(G). \label{eq-3-1-2}
	\end{eqnarray}
	Moreover, by Lemma \ref{lm-1} (ii) we have $\mu_1(G)\leq p$.
    Now, bearing in mind that $p\leq q$ and using Lemma \ref{lm-1} (iii) we obtain
	$$\mu(G\vee qK_1)=(p+q, q+\mu_1(G), \cdots, q+\mu_{p-1}(G), p, \cdots, p, 0).$$
	We next consider the following two cases:
	
	\textbf{Case 1.} $1\leq k\leq p-1$.
	
	If $e(G)\leq \binom{k+1}{2}$, then we have
	\begin{eqnarray}
        s_{k-1}(G)\leq 2e(G)\leq e(G)+\binom{k+1}{2},\label{eq-3-1-3}
    \end{eqnarray}
	which, together with (\ref{eq-3-1-2}) and the assumption $p\leq q$, yields that
	\begin{eqnarray}
		s_{k}(G\vee qK_1)&=&p+kq+s_{k-1}(G)\nonumber\\
		&\leq & (k+1)q+e(G)+\binom{k+1}{2} \label{eq-3-1-4} \\
		&\leq & e(G\vee qK_1)+\binom{k+1}{2}.\nonumber
	\end{eqnarray}
    Furthermore, if $s_{k}(G\vee qK_1)=e(G\vee qK_1)+\binom{k+1}{2}$, then equalities hold in all the inequalities in (\ref{eq-3-1-3}) and (\ref{eq-3-1-4}),
    which yields that $p=k+1$, $e(G)=\binom{p}{2}$, and $s_{p-2}(G)=(p-1)p$.
    However, $s_{p-2}(G)\leq (p-2)\mu_1(G)<(p-1)p$ makes a contradiction.
    We thus conclude that for $1\leq k\leq p-1$ and $e(G)\leq \binom{k+1}{2}$,
	$$s_{k}(G\vee qK_1)<e(G\vee qK_1)+\binom{k+1}{2}.$$
	
	If $e(G)>\binom{k+1}{2}$, noting that $1\leq k<p\leq q$, we obtain
	\begin{eqnarray*}
		s_{k}(G\vee qK_1)&\leq&k(p+q)\\
        &=& pq+\binom{k+1}{2}+\binom{k+1}{2}-\big(pq-k(p+q)+k^2+k\big)\\
		&<& pq+e(G)+\binom{k+1}{2}-(p-k)(q-k)-k\\
		&<& e(G\vee qK_1)+\binom{k+1}{2}.
	\end{eqnarray*}
	
	\textbf{Case 2.} $p\leq k\leq p+q-1$.
	
	Noting that $s_{p-1}(G)=2e(G)$ and $e(G)\leq \binom{p}{2}$, we have
	\begin{eqnarray}
		s_{k}(G\vee qK_1)&=& pq+(k-p+1)p+2e(G) \nonumber \\
		&\leq& e(G\vee qK_1)+\binom{k+1}{2}+(k-p+1)p+\binom{p}{2}-\binom{k+1}{2} \nonumber \\
		&=& e(G\vee qK_1)+\binom{k+1}{2}-\frac{(k-p+1)(k-p)}{2} \nonumber \\
		&\leq& e(G\vee qK_1)+\binom{k+1}{2}. \label{eq-3-1-5}
	\end{eqnarray}

	Furthermore, it is easy to check that if $s_{k}(G\vee qK_1)=e(G\vee qK_1)+\binom{k+1}{2}$,
	then equalities hold in all the inequalities in (\ref{eq-3-1-5}),
    which yields that $k=p$ and $e(G)=\binom{p}{2}$, that is, $k=p$ and $G\cong K_p$.

    Conversely, if $k=p$ and $G\cong K_p$, then $G\vee qK_1$ is a threshold graph on $p+q$ vertices with clique number $p+1=k+1$.
    From Theorem \ref{thm-split} it follows that $s_{k}(G\vee qK_1)=e(G\vee qK_1)+\binom{k+1}{2}$.
	
	This completes the proof of Theorem \ref{thm-3-1-1}.
\end{proof}

\subsection{$c$-cyclic graphs with $c\in\{0,1,2\}$}

Recall that Brouwer's conjecture has been proven to be true for $c$-cyclic graphs with $c\in\{0,1,2\}$.
In this subsection, we do the same thing for the full Brouwer's conjecture.
To this end, we first establish an auxiliary result,
which refines Lemmas 2.6 and 2.7 in \cite{Wang} and Lemma 4.7 in \cite{Du}.

\begin{lemma}\label{lm-3-2-1}
Let $e_1,e_2,\cdots,e_t$ be some edges in a connected graph $G$ of order $n$ such that $G-\{e_1,e_2,\cdots,e_t\}:=G_1\cup G_2$,
where $G_1$ and $G_2$ are two vertex-disjoint graphs of order $n_1\geq 3$ and $n_2\geq 2$, respectively.
Let $k$ be an integer with $3\leq k\leq n-2$ and $n^\prime_i=\min\{n_i-1,k\}$.
Suppose for any $\ell_i\in\{1,2,\cdots,n_i^\prime\}$,
$s_{\ell_i}(G_i)\leq e(G_i)+\binom{\ell_i+1}{2},\, i=1,2$.
Then $s_k(G)<e(G)+\binom{k+1}{2}$ if one of the following conditions holds:

(i) $s_k(G_1\cup G_2)=s_k(G_1)$ or $s_k(G_2)$, $e(G_1)\geq t$ and $e(G_2)\geq t$;

(ii) $s_k(G_1\cup G_2)=s_{k_1}(G_1)+s_{k_2}(G_2)$ and $t\leq k-1$, where $k_1+k_2=k$ and $k_1k_2\neq 0$;

(iii) $t\leq k-1$, $e(G_1)\geq t$ and $e(G_2)\geq t$;

(iv) $t=k$, $G_1\cong S_{n_1}$ and $G_2\cong S_{n_2}$, $T_{n_2}^2$ ($n_2\geq 6$), or $T_{n_2}^3$ ($n_2\geq 7$);
in this case, $G$ is a $(t-1)$-cyclic graph.
\end{lemma}

\begin{proof}
Since $G$ is connected and $G_1\cup G_2$ is not, we have $\mu_{n-1}(G)>0$ and $\mu_{n-1}(G_1\cup G_2)=0$.
Moreover, by Lemma \ref{lm-1} (iv) we get
$$\mu_i(G_1\cup G_2)\leq \mu_i(G),\, i=1,2,\cdots,n-2.$$
These, together with the fact that $s_{n-1}(G)=2e(G)=2(e(G_1\cup G_2)+t)=s_{n-1}(G_1\cup G_2)+2t$, would yield that for $3\leq k\leq n-2$,
\begin{eqnarray}
s_k(G)&=&s_k(G_1\cup G_2)+2t+\sum_{i=k+1}^{n-1}\big(\mu_i(G_1\cup G_2)-\mu_i(G)\big)\nonumber\\
&\leq& s_k(G_1\cup G_2)+2t-\mu_{n-1}(G)<s_k(G_1\cup G_2)+2t. \label{eq-3-2-1}
\end{eqnarray}

{\bf (i)} If $s_k(G_1\cup G_2)=s_k(G_1)$ and $e(G_i)\geq t$ for $i=1,2$,
then $k\leq n_1-1$, and hence the hypothesis gives that $s_k(G_1)\leq e(G_1)+\binom{k+1}{2}$.
Now, by (\ref{eq-3-2-1}) we can obtain
\begin{eqnarray*}
s_k(G)<s_k(G_1)+2t\leq e(G_1)+\binom{k+1}{2}+2t\leq e(G)+\binom{k+1}{2}.
\end{eqnarray*}
By the same argument, we can also obtain the desired result if $s_k(G_1\cup G_2)=s_k(G_2)$ and $e(G_i)\geq t$ for $i=1,2$.

{\bf (ii)} If $s_k(G_1\cup G_2)=s_{k_1}(G_1)+s_{k_2}(G_2)$ and $t\leq k-1$, where $k_1+k_2=k$ and $k_1k_2\neq 0$,
then for $i=1,2$, we have $1\leq k_i\leq k-1$ and $k_i\leq n_i-1$,
and hence the hypothesis gives that $s_{k_i}(G_i)\leq e(G_i)+\binom{k_i+1}{2}$.
Again by (\ref{eq-3-2-1}) we obtain
\begin{eqnarray*}
s_k(G)&<&s_{k_1}(G_1)+s_{k_2}(G_2)+2t\\
&\leq& e(G_1)+\binom{k_1+1}{2}+e(G_2)+\binom{k_2+1}{2}+2t\\
&\leq& e(G)+\frac{k_1^2+k_2^2+k_1+k_2}{2}+t\\
&\leq& e(G)+\binom{k+1}{2}-k_1k_2+(k-1)\\
&=& e(G)+\binom{k+1}{2}+(k_1-1)(k_1-(k-1))\leq e(G)+\binom{k+1}{2}.
\end{eqnarray*}

{\bf (iii)} If $t\leq k-1$, $e(G_1)\geq t$ and $e(G_2)\geq t$, then combining the arguments of (i) and (ii),
we can obtain the desired result.

{\bf (iv)} Suppose that $t=k$ and $G_1\cong S_{n_1}$ ($n_1\geq 3$).
It is well known that
$$\mu(G_1)=\{n_1,1,\cdots,1,0\}.$$

If $G_2\cong S_{n_2}$ ($n_2\geq 2$), then
$$\mu(G_1\cup G_2)=\{n_1,n_2,1,\cdots,1,0,0\}.$$
Thus, since
$e(G)=e(G_1\cup G_2)+t=n_1+n_2-2+t$,
it follows from (\ref{eq-3-2-1}) that for $3\leq k\leq n-2$,
\begin{eqnarray*}
s_k(G)&<&s_k(G_1\cup G_2)+2t\\
&=&n_1+n_2+(k-2)+2t\\
&=&e(G)+k+t\leq e(G)+\binom{k+1}{2}.
\end{eqnarray*}

If $G_2\cong T_{n_2}^2$ ($n_2\geq 6$), by Lemma \ref{lm-2} we have $1<\mu_2(T_{n_2}^2)<2.7$.
Moreover, as shown in the proof of Lemma 4.6 in \cite{Du}, the characteristic polynomial of $L(T_{n_2}^2)$ is
$$\phi(T_{n_2}^2;x):=\det(xI-L(T_{n_2}^2))=x(x-1)^{n_2-6}g_1(x),$$
where
$$g_1(x):=x^5-(n_2+4)x^4+(6n_2-1)x^3-(11n_2-14)x^2+(6n_2-5)x-n_2.$$
For $n_2\geq 7$, one can directly check by computer that
$g_1(n_2-1.7)\approx 0.3(n_2-6.01366)(n_2-4.31803)(n_2-2.08197)(n_2-0.386343)>0$,
which implies that $\mu_1(T_{n_2}^2)<n_2-1.7$.
Thus, from (\ref{eq-3-2-1}) it follows that for $3\leq k\leq n-2$,
\begin{eqnarray*}
s_k(G)&<&s_k(G_1\cup G_2)+2t\\
&<&n_1+n_2-1.7+2.7(k-2)+2t\\
&=&e(G)+2.7k-5.1+t=e(G)+\binom{k+1}{2}-\frac{(k-3.4)(k-3)}{2}\\
&\leq& e(G)+\binom{k+1}{2}\,\, (\textrm{as $k$ is an integer at least 3}).
\end{eqnarray*}
For $n_2=6$, a direct calculation shows that
$$\mu(T_{n_2}^2)=\{4.30278, 2.61803, 2.00000, 0.69722, 0.38197, 0\},$$
from which we can also deduce that $s_k(G)<n_1+n_2-1.65+2.65(k-2)+2t\leq n_1+n_2-1.7+2.7(k-2)+2t\leq e(G)+\binom{k+1}{2}$, as desired.

If $G_2\cong T_{n_2}^3$ ($n_2\geq 7$), again by Lemma \ref{lm-2} we see that $1<\mu_2(T_{n_2}^3)<2.7$.
Moreover, for $n_2=7$, a direct calculation gives that $\mu_1(T_{n_2}^3)\approx4.41421<n_2-1.7$,
while for $n_2\geq 8$, as shown in the proof of Lemma 4.6 in \cite{Du}, the characteristic polynomial of $L(T_{n_2}^3)$ is
$$\phi(T_{n_2}^3;x):=\det(xI-L(T_{n_2}^3))=x(x-1)^{n_2-8}g_2(x),$$
where
\begin{eqnarray*}
g_2(x)&:=&x^7-(n_2+6)x^6+(9n_2+3)x^5-(30n_2-42)x^4\\
&&+(45n_2-87)x^3-(30n_2-48)x^2+(9n_2-8)x-n_2.
\end{eqnarray*}
We can now check directly by computer that
$g_2(n_2-2)=(n_2^2-7n_2+11)^2(n_2^2-7n_2+8)>0$,
which implies that $\mu_1(T_{n_2}^3)<n_2-2<n_2-1.7$.
Thus, by the same argument as above for the case of $G_2\cong T_{n_2}^2$,
we get the desired result.

This completes the proof of Lemma \ref{lm-3-2-1}.
\end{proof}

For convenience, we will say a $t$-edge-cut $Q$ in a connected graph $G$ {\it nice}
if $G-Q:=G_1\cup G_2$ with $G_1$ and $G_2$ being two vertex-disjoint graphs with $e(G_1)\geq t$ and $e(G_2)\geq t$.

We are now ready to present the main result of this subsection.

\begin{theorem}\label{thm-3-2-2}
Let $G$ be a $c$-cyclic graph on $n$ vertices, where $c\in\{0,1,2\}$.
Then for each $k\in\{1,2,\cdots,n-1\}$
\begin{eqnarray}
s_k(G)\leq e(G)+\binom{k+1}{2}.\label{eq-3-2-2}
\end{eqnarray}
Moreover, the equality holds in (\ref{eq-3-2-2}) if and only if $k=1$ and $G\cong S_n$ for $c=0$, or
$k=2$ and $G\cong G(n-3,1)$ for $c=1$, or $k=2$ and $G\cong G(n-4,2)$ for $c=2$.
\end{theorem}

\begin{proof}
By Lemma \ref{lm-7}, it suffices to prove that for $3\leq k\leq n-2$,
\begin{eqnarray}
	s_{k}(G)<e(G)+\binom{k+1}{2}. \label{eq-3-2-3}
\end{eqnarray}
Indeed, by Lemma \ref{lm-5} and the fact that $e(G)=n-1+c$, we obtain
\begin{eqnarray}
	s_{k}(G)<2e(G)-n+2k=e(G)+\binom{k+1}{2}-\frac{k(k-3)}{2}+c-1. \label{eq-3-2-4}
\end{eqnarray}

If $c=0$ or $1$, then for $3\leq k\leq n-2$, (\ref{eq-3-2-3}) follows directly from (\ref{eq-3-2-4}).

If $c=2$, then for $4\leq k\leq n-2$, (\ref{eq-3-2-3}) follows again from (\ref{eq-3-2-4}).
Consequently, it remains to show that $s_{3}(G)<e(G)+6$.
This, in fact, has been confirmed to be true for all graphs with at most 9 vertices \cite{Li}.
We now assume to the contradiction that
\begin{eqnarray}
s_{3}(G)\geq e(G)+6,\label{eq-3-2-5}
\end{eqnarray}
such that $G$ has as few vertices as possible.
This implies that $|V(G)|=n\geq 10$
and that $s_3(H)<e(H)+6$ for any 2-cyclic graph $H$ with $|V(H)|<|V(G)|$.
Moreover, we have the following fact:

\textbf{Fact 1.} Every cut-edge $uv$ in $G$ (if exists) is a pendant edge (i.e., $d_G(u)\geq 2$ and $d_G(v)=1$).

Otherwise, we have $G-\{uv\}:=G_1\cup G_2$, where $G_1$ is a 2-cyclic graph with $n>|V(G_1)|\geq 4$ and $e(G_1)\geq 5$
and $G_2$ is a 0-cyclic graph with $|V(G_2)|\geq 2$ and $e(G_2)\geq 1$,
or $G_1$ and $G_2$ are both 1-cyclic graph with $|V(G_i)|\geq 3$ and $e(G_i)\geq 3$ ($i=1,2$);
this means that $\{uv\}$ is a nice 1-edge-cut.
Furthermore, for the former case, the minimality of $G$ and Lemma \ref{lm-7} yield that
$s_{\ell_1}(G_1)\leq e(G_1)+\binom{\ell_1+1}{2}$ for $\ell_1\in\{1,2,3\}$.
We also have $s_{\ell_2}(G_2)\leq e(G_2)+\binom{\ell_2+1}{2}$ for $\ell_2\in\{1,2,\cdots,|V(G_2)|-1\}$,
which is just proven previously (for $c=0$).
Consequently, by Lemma \ref{lm-3-2-1} (iii) we obtain $s_{3}(G)< e(G)+6$, contradicting the assumption (\ref{eq-3-2-5}).
For the latter case, we have
\begin{eqnarray}
s_{\ell_i}(G_i)\leq e(G_i)+\binom{\ell_i+1}{2}\,\, \textrm{for}\,\, \ell_i\in\{1,2,\cdots,|V(G_i)|-1\}, i=1,2, \label{eq-3-2-6}
\end{eqnarray}
which is also just shown previously (for $c=1$).
Lemma \ref{lm-3-2-1} (iii) now yields the same contradiction as the former case.
Therefore, Fact 1 follows.
\begin{figure}
  \centering
  \includegraphics[width=3.5cm]{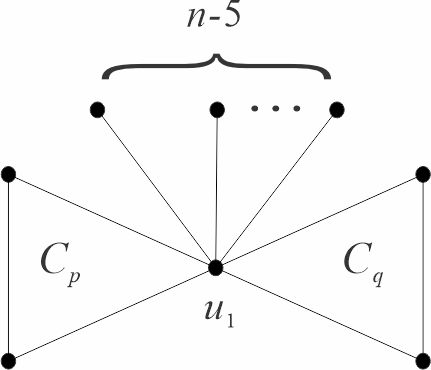}\\
  \caption{The 2-cyclic graph $B_0$.}\label{fig6}
\end{figure}

We next consider the following two cases according to the base $\widehat{G}$ of $G$.

\textbf{Case 1.} $\widehat{G}\cong \infty(p,l,q)$, where $p\geq q\geq 3$ and $l\geq 1$ (see Figure \ref{fig2}).

By Fact 1, we see that $G$ has no non-pendant cut-edges, and hence $l=1$.
We further claim that $p=q=3$.
Indeed, if $p\neq 3$ or $q\neq 3$ (without loss of generality assume that $p\geq 4$), then
let $e_1$ and $e_2$ be the edges on $C_p$ incident to the vertex $u_1$.
It is easy to check that $\{e_1,e_2\}$ is a nice 2-edge-cut and $G-\{e_1,e_2\}:=G_1\cup G_2$ with
$G_1$ and $G_2$ being 1-cyclic graph and 0-cyclic graph, respectively.
As above, (\ref{eq-3-2-6}) still holds, which, together with Lemma \ref{lm-3-2-1} (iii), yields the same contradiction.
Our claim follows.

Recall that $G$ can be obtained by attaching trees to some vertices of its base $\widehat{G}$.
Fact 1 here tells us that these trees must be stars with their centers identifying some vertices of $\widehat{G}$.
Furthermore, except the vertex $u_1$, if there is a star attached to any other vertex, without loss of generality, of $C_p$ in $\widehat{G}$,
then let $e_1$ and $e_2$, as above, be the edges on $C_p$ incident to $u_1$.
Clearly, $\{e_1,e_2\}$ is also a nice 2-edge-cut and $G-\{e_1,e_2\}:=G_1\cup G_2$ with
$G_1$ and $G_2$ being 1-cyclic graph and 0-cyclic graph, respectively.
Thus, the desired contradiction will appear after using the same argument as above.
This means that all the stars can only be attached to $u_1$ of $\widehat{G}$,
and hence $G$ is the graph $B_0$ as shown in Figure \ref{fig6}.
However, by Grone-Merris-Bai theorem, we have
$s_3(G)\leq n+5+1=e(G)+5<e(G)+6$, a contradiction too.

\textbf{Case 2.} $\widehat{G}\cong \theta(p,l,q)$, where $p\geq q\geq l\geq 1$ (see Figure \ref{fig2}).

We first claim that $p\leq 3$, which implies that all possible bases of $G$ are
$\theta(2,1,2)$, $\theta(2,2,2)$, $\theta(3,1,2)$, $\theta(3,1,3)$, $\theta(3,2,2)$, $\theta(3,2,3)$ and $\theta(3,3,3)$ (as shown in Figure \ref{fig7}).
Indeed, if $p\geq 4$, then $\{xu_1, u_{p-1}y\}$ is a nice 2-edge-cut
and $G-\{xu_1, u_{p-1}y\}:=G_1\cup G_2$ with $G_1$ and $G_2$ being 1-cyclic graph and 0-cyclic graph, respectively.
The same argument as used in Case 1 yields the desired contradiction. Our claim follows.

As shown in Case 1, $G$ can be obtained by identifying the centers of stars with some vertices of $\widehat{G}$.
We further claim that these vertices of $\widehat{G}$ can only be $x$, $y$, and the middle vertex $z$ of each $x$-$y$ path of length 2
(in this case, there is at most one $S_2$ attached to $z$).
Otherwise, there are always two edges $e_1, e_2$ in $\widehat{G}$ such that
$\{e_1,e_2\}$ is a nice 2-edge-cut and $G-\{e_1,e_2\}:=G_1\cup G_2$ with
$G_1$ and $G_2$ being 1-cyclic graph and 0-cyclic graph, respectively.
This, as above, will lead to a contradiction. Our claim follows.

Now, by combining the above arguments, we can conclude that
all possible $G$ are the graphs $B_i$ ($i=1,2,\cdots,16$) as shown in Figure \ref{fig7},
where we assume without loss of generality that $a\geq b\geq 0$.
Moreover, since $n\geq 10$, we can easily check that $a\geq 2$ for $B_i$ with $i\neq 7,13,15,16$, and $a\geq 1$ otherwise.
We will show in the following that $s_3(B_i)<e(B_i)+6$ for $i=1,2,\cdots,16$, which also contradicts with the assumption (\ref{eq-3-2-5}).
This, together with the above discussions, suggests that under the assumption (\ref{eq-3-2-5}), each case leads to a contradiction,
which means that the assumption (\ref{eq-3-2-5}) is wrong, and hence (\ref{eq-3-2-3}) follows, completing the proof.
\begin{figure}
  \centering
  \includegraphics[width=15.5cm]{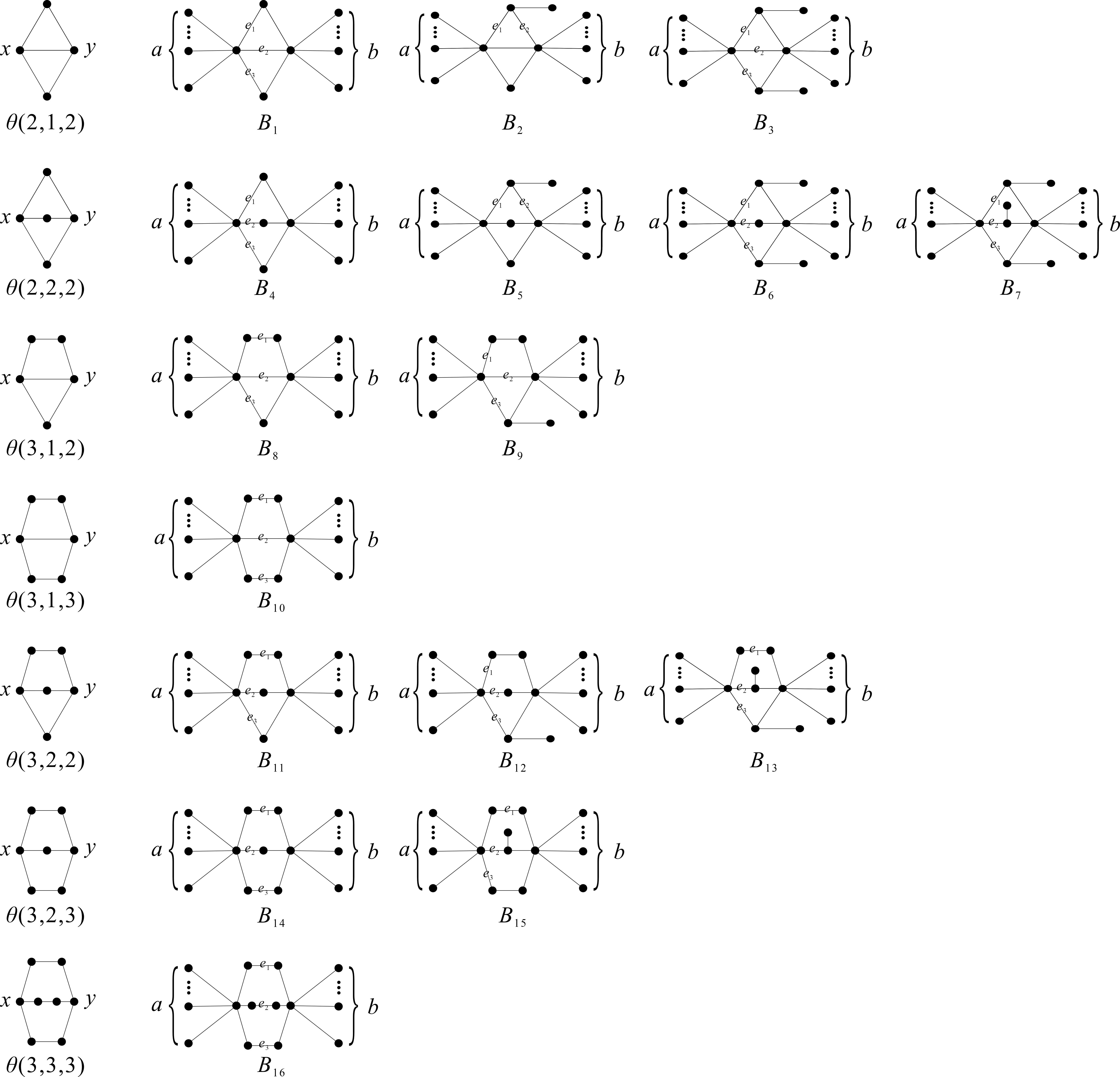}\\
  \caption{The 2-cyclic graphs $B_i$ ($i=1,2,\cdots,16$) and corresponding bases.}\label{fig7}
\end{figure}

$\bullet$ For $i=1,4,8,10,11,14,16$, let $e_1,e_2,e_3$ be the edges in $B_i$ as marked in Figure \ref{fig7}.
It is easy to see that $B_i-\{e_1,e_2,e_3\}=S_{n_1}\cup S_{n_2}$, where $n_1\geq 3$ and $n_2\geq 3$.
Thus, Lemma \ref{lm-3-2-1} (iv) yields that $s_3(B_i)<e(B_i)+6$, as desired.

$\bullet$ For $i=6,12,13,15$, as above, we have $B_i-\{e_1,e_2,e_3\}=S_{n_1}\cup T^2_{n_2}$, where $n_1\geq 3$ and $n_2\geq 6$.
Thus, Lemma \ref{lm-3-2-1} (iv) yields the desired result.

$\bullet$ For $i=7$, if $a=1$ (in this case we have $n=10)$, a direct calculation shows that $s_3(B_7)\approx 6.1926+4.3028+3.4142<e(B_7)+6$. 
If $a\geq 2$, since $B_7-\{e_1,e_2,e_3\}=S_{n_1}\cup T^3_{n_2}$ with $n_1\geq 3$ and $n_2\geq 7$,
by Lemma \ref{lm-3-2-1} (iv) we also have $s_3(B_7)<e(B_7)+6$.
\begin{figure}
  \centering
  \includegraphics[width=8.5cm]{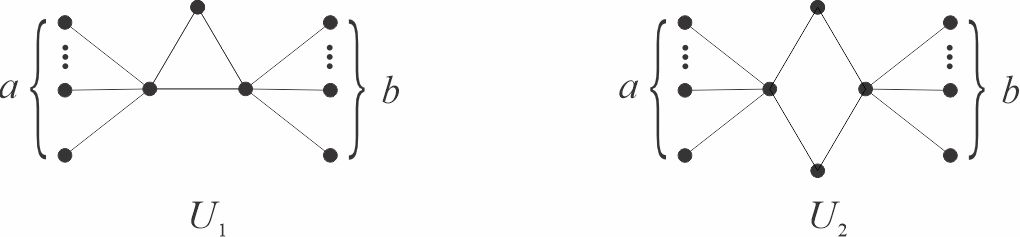}\\
  \caption{The 1-cyclic graphs $U_1$ and $U_2$ with $a\geq b\geq 0$.}\label{fig8}
\end{figure}

$\bullet$ For $i=2,5$, let $e_1, e_2$ be the edges in $B_i$ as marked in Figure \ref{fig7}.
Clearly, $B_2-\{e_1,e_2\}=U_1\cup S_2$ and $B_5-\{e_1,e_2\}=U_2\cup S_2$,
where $U_1$ and $U_2$ are the graphs as shown in Figure \ref{fig8}.
It was shown in \cite{Du} that $\mu_3(U_j)\leq 2$ for $j=1,2$ if $n\geq 8$
\footnote{From the proof of Lemma 4.2 in \cite{Du}, the condition $n\geq 9$ can be easily improved to $n\geq 8$.}
and $a\geq b\geq 0$.
This, together with the fact that $\mu_1(S_2)=2$, yields that
$s_3(U_j\cup S_2)=s_2(U_j)+s_1(S_2)$ for $j=1,2$.
Moreover, for $j=1,2$ we have $s_3(U_j)\leq e(U_j)+6$ (since $U_j$ is a 1-cyclic graph) and $s_1(S_2)=e(S_2)+1$.
Thus, by Lemma \ref{lm-3-2-1} (ii) we can obtain $s_3(B_i)<e(B_i)+6$ for $i=2,5$.

\begin{figure}
  \centering
  \includegraphics[width=8.5cm]{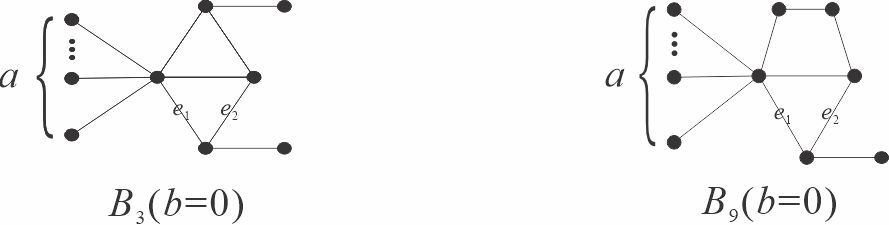}\\
  \caption{The 2-cyclic graphs $B_3$ and $B_9$ with $b=0$.}\label{fig9}
\end{figure}

$\bullet$ For $i=3,9$, if $b\geq 1$, then $B_i-\{e_1,e_2,e_3\}=S_{n_1}\cup T^2_{n_2}$ with $n_1\geq 3$ and $n_2\geq 6$;
as the case of $i=6,12,13,15$, the desired result follows.
If $b=0$, then $B_3-\{e_1,e_2\}:=U_1\cup S_2$ and $B_9-\{e_1,e_2\}:=U_2\cup S_2$,
where $e_1$ and $e_2$ are the edges in $B_i$ ($i=3,9$) as marked in Figure \ref{fig9};
as the case of $i=2,5$, we also have the desired result.
\end{proof}

\subsection{Nordhaus-Gaddum-type results for $s_k$}

In this subsection, we will consider a Nordhaus-Gaddum version of the full Brouwer's conjecture (i.e., Conjecture \ref{conj-4})
and give some partial solutions to it. For the purpose, we first present a simple but useful lemma.

\begin{lemma}\label{thm-3}
	Let $G$ be a graph of order $n$ and $\overline{G}$ be its complement. Then for $1\leq k\leq (n-1)/2$,
	\begin{eqnarray}
		s_k(G)+s_k(\overline{G})\leq \binom{n}{2}+2\cdot\binom{k+1}{2}\label{eq-3-3-1}
	\end{eqnarray}
	if and only if
	\begin{eqnarray}
		s_{n-(k+1)}(G)+s_{n-(k+1)}(\overline{G})\leq \binom{n}{2}+2\cdot\binom{n-k}{2}.\label{eq-3-3-2}
	\end{eqnarray}
    Moreover, the equality holds in (\ref{eq-3-3-1}) if and only if the equality holds in (\ref{eq-3-3-2}).
\end{lemma}

\begin{proof}
	For any integer $j$ with $1\leq j\leq n-1$, by Lemma \ref{lm-1} (i) we have
	$$\mu_j(G)+\mu_{n-j}(\overline{G})=\mu_j(G)+n-\mu_{j}(G)=n.$$
	Thus, for $1\leq k\leq (n-1)/2$, we obtain
	\begin{eqnarray}\label{eq-3}
		s_k(G)+s_k(\overline{G})
		&=&s_{n-(k+1)}(G)+s_{n-(k+1)}(\overline{G})-\sum_{i=k+1}^{n-(k+1)}\big[\mu_i(G)+\mu_i(\overline{G})\big]\nonumber\\
		&=&s_{n-(k+1)}(G)+s_{n-(k+1)}(\overline{G})-\sum_{i=k+1}^{n-(k+1)}\big[\mu_i(G)+\mu_{n-i}(\overline{G})\big]\nonumber\\
		&=&s_{n-(k+1)}(G)+s_{n-(k+1)}(\overline{G})-\big[n-(2k+1)\big]n.
	\end{eqnarray}
	The desired result now follows from (\ref{eq-3}) and the fact that
	$$\big[n-(2k+1)\big]n+2\cdot\binom{k+1}{2}=2\cdot\binom{n-k}{2},$$
	completing the proof.
\end{proof}

We are now ready to give the main results of this subsection.

\begin{theorem}\label{thm-4}
	Let $G$ be a graph of order $n$ and $\overline{G}$ be its complement.
    Define $\phi(n):=\sqrt{2n^2-2n+1}$. Then for $1\leq k\leq n-\frac{\phi(n)+1}{2}$ or $\frac{\phi(n)-1}{2}\leq k\leq n-2$,
    \begin{eqnarray}
    s_k(G)+s_k(\overline{G})<\binom{n}{2}+2\cdot\binom{k+1}{2}.\label{eq-3-3-3}
    \end{eqnarray}
\end{theorem}

\begin{proof}
	By Lemma \ref{thm-3}, it suffices to show that the inequality (\ref{eq-3-3-3}) is true for $\frac{\phi(n)-1}{2}\leq k\leq n-2$.
	Indeed, if $G$ is connected, then $\mu_{n-1}(G)>0$, and hence $s_k(G)\leq 2e(G)-\mu_{n-1}(G)<2e(G)$,
    while if $G$ is disconnected, then $\overline{G}$ must be connected, and hence $s_k(\overline{G})< 2e(\overline{G})$.
    Consequently, we have
	$$s_k(G)+s_k(\overline{G})<2e(G)+2e(\overline{G})=2\cdot\binom{n}{2}\leq \binom{n}{2}+2\cdot\binom{k+1}{2}$$
	if $k^2+k-\binom{n}{2}\geq 0$, which is always true when $\frac{\phi(n)-1}{2}\leq k\leq n-2$, completing the proof.
\end{proof}

\begin{theorem}\label{thm-5}
	Let $G$ be a graph on $n$ vertices.
	If $e(G)\leq \frac{2-\sqrt{2}}{4}\binom{n}{2}$ or $e(G)\geq \frac{2+\sqrt{2}}{4}\binom{n}{2}$,
	then the inequality (\ref{eq-3-3-3}) holds for each $k\in\{1,2,\cdots,n-2\}$.
\end{theorem}

\begin{proof}
	For convenience, let $m:=e(G)$ and $\overline{m}:=e(\overline{G})$.
	Clearly, $m+\overline{m}=\binom{n}{2}$.
	Furthermore, since $m\leq \frac{2-\sqrt{2}}{4}\binom{n}{2}$ or $m\geq \frac{2+\sqrt{2}}{4}\binom{n}{2}$,
	we have
	\begin{eqnarray}
		m^2+\overline{m}^2\geq \frac{3}{4}\binom{n}{2}^2,\label{eq-4}
	\end{eqnarray}
	which follows from the fact that the function $f(x):=x^2+(1-x)^2\geq \frac{3}{4}$
	when $0\leq x\leq \frac{2-\sqrt{2}}{4}$ or $\frac{2+\sqrt{2}}{4}\leq x\leq 1$.
	Now, by Lemma \ref{thm-3}, we just need to show that the inequality (\ref{eq-3-3-3}) is true for $1\leq k\leq (n-1)/2$.
	Indeed, by Lemma \ref{lm-6} and the arithmetic-geometric mean inequality, we obtain
	\begin{eqnarray*}
		&&s_k(G)+s_k(\overline{G})\nonumber\\
		&\leq&\frac{2k(m+\overline{m})+\sqrt{k(n-k-1)}\big[\sqrt{n(n-1)m-2m^2}+\sqrt{n(n-1)\overline{m}-2\overline{m}^2}\big]}{n-1}\nonumber\\
		&\leq& kn+\frac{\sqrt{k(n-k-1)}\sqrt{2n(n-1)(m+\overline{m})-4(m^2+\overline{m}^2)}}{n-1}\nonumber\\
		&\leq& kn+\frac{\sqrt{k(n-k-1)}}{2}n. \quad (\textrm{by (\ref{eq-4})})
	\end{eqnarray*}
	Thus, we can conclude that
	$$s_k(G)+s_k(\overline{G})<\binom{n}{2}+2\cdot\binom{k+1}{2}$$
	if $kn+\frac{\sqrt{k(n-k-1)}}{2}n<\binom{n}{2}+2\cdot\binom{k+1}{2}$, which is equivalent to $h_1(k)>0$,	where
	\begin{eqnarray*}
		h_1(x):&=&4x^4-8(n-1)x^3+(3n-2)^2x^2-(5n-4)n(n-1)x+n^2(n-1)^2\\
		&=&4x^2\big[x-(n-1)\big]^2+(5n-4)n\big[x-(n-1)/2\big]^2-n(n-4)(n-1)^2/4.
	\end{eqnarray*}
	It is easy to see that $h_1(x)$ is a decreasing function when $1\leq x\leq (n-1)/2$.
	Consequently, for $1\leq k\leq (n-1)/2$, we have
	$$h_1(k)\geq h_1((n-1)/2)=(2n+1)(n-1)^2/4>0,$$
	as desired.
	This completes the proof of Theorem \ref{thm-5}.
\end{proof}

\begin{theorem}\label{thm-6}
	Let $G$ be a graph on $n$ vertices with maximum degree $\Delta$ and minimum degree $\delta$.
	For a given real number $t>2$, if $n\geq \frac{11t^2+8t+4}{(t-2)^2}$ and $\Delta-\delta\leq n/t$,
	then the inequality (\ref{eq-3-3-3}) holds for each $k\in\{1,2,\cdots,n-2\}$.
\end{theorem}

\begin{proof}
	Let $d_i$ and $\overline{d}_i$ denote the $i$-th largest degrees of $G$ and $\overline{G}$, respectively.
	Set $\lambda_i:=\lambda_i(A(G))$ and $\overline{\lambda}_i:=\lambda_i(A(\overline{G}))$ for convenience.
	By Lemma \ref{thm-3}, we just need to prove the inequality (\ref{eq-3-3-3}) to be true for $1\leq k\leq (n-1)/2$.
	Indeed, by applying Lemma \ref{lm-3} to $L(G)=D(G)-A(G)$ and $L(\overline{G})=D(\overline{G})-A(\overline{G})$, respectively, we have,
	\begin{eqnarray*}
		s_k(G)&\leq& \sum_{i=1}^k d_i+\sum_{i=n}^{n-k+1}(-\lambda_i)\leq k\Delta+\sum_{i=1}^{k}|\lambda_{n-i+1}|;\\
		s_k(\overline{G})&\leq& \sum_{i=1}^k \overline{d}_i+\sum_{i=n}^{n-k+1}(-\overline{\lambda}_i)
		\leq k(n-1-\delta)+\sum_{i=1}^{k}|\overline{\lambda}_{n-i+1}|.
	\end{eqnarray*}
	Thus, by Lemma \ref{lm-8}, we obtain
	\begin{eqnarray*}
		s_k(G)+s_k(\overline{G})&\leq& k(n-1+\Delta-\delta)+\sqrt{2k}(n/2+k)\\
		&\leq& k(n-1+n/t)+\sqrt{2k}(n/2+k).
	\end{eqnarray*}
	Now, we can conclude that
	$$s_k(G)+s_k(\overline{G})<\binom{n}{2}+2\cdot\binom{k+1}{2}$$
    if $k(n-1+\frac{n}{t})+\sqrt{2k}(\frac{n}{2}+k)<\binom{n}{2}+2\cdot\binom{k+1}{2}$,
    which is equivalent to $\frac{h_2(k)}{t^2}>0$, where
	\begin{eqnarray*}
		h_2(x):&=&4t^2x^4-8t(nt+n-t)x^3+4\big[(2t^2+2t+1)n^2-(7t^2+4t)n+4t^2\big]x^2\\
		&&-2nt\big[2(t+1)n^2-(5t+2)n+4t\big]x+n^2(n-1)^2t^2.
	\end{eqnarray*}
	We further claim that $h_2(x)$ is a decreasing function when $1\leq x\leq (n-1)/2$,
	from which we can derive that when $n\geq \frac{11t^2+8t+4}{(t-2)^2}$ and $1\leq k\leq (n-1)/2$,
	\begin{eqnarray*}
		h_2(k)&\geq& h_2((n-1)/2)\\
		&=&\frac{1}{4}(n-1)\big[(t-2)^2n^3-(11t^2+8t+4)n^2+(19t^2+12t)n-13t^2\big]>0,
	\end{eqnarray*}
	and thus complete the proof of Theorem \ref{thm-6}.
	
	Indeed, under the conditions that $n\geq \frac{11t^2+8t+4}{(t-2)^2}$ and $1\leq x\leq (n-1)/2$,
	one can check that
	$$h_2^{\prime\prime\prime}(x)=48t\big[2tx-(nt+n-t)\big]<0,$$
	which implies that $h_2^{\prime\prime}(x)$ is a decreasing function with respect to $x$ and thus,
	\begin{eqnarray*}
		h_2^{\prime\prime}(x)\geq h_2^{\prime\prime}((n-1)/2)=4\big[(t^2-2t+2)n^2-(8t^2+2t)n+5t^2\big]>0,
	\end{eqnarray*}
	which again implies that $h_2^{\prime}(x)$ is an increasing function with respect to $x$ and consequently,
	\begin{eqnarray*}
		h_2^{\prime}(x)&\leq& h_2^{\prime}((n-1)/2)\\
		&=&-2\big[(t-2)n^3+(7t^2+4t+2)n^2-(12t^2+5t)n+6t^2\big]<0.
	\end{eqnarray*}
	Our claim follows.
\end{proof}

\section{Concluding remark}

We have given a concise version of the full Brouwer's conjecture and
have proven it to be true for a complete split graph as well as two families of its spanning subgraphs:
one is the split graph, which can be obtained from a complete split graph by removing some edges between the clique and the independent set; 
the other is the graph obtained from a complete split graph by removing some edges whose endpoints are both in the clique. 
Note that there is an additional condition $p\leq q$ imposed on the latter case (see Theorem \ref{thm-3-1-1}),
which seems difficult to be removed.
In fact, as shown in the following result, the case of $p>q=1$ is as hard as the whole conjecture.

\begin{lemma} \label{thm-4-1}
	With a convention that $s_{0}(G)=0$ and $\binom{1}{2}=0$,
    the full Brouwer's conjecture holds for a graph $G$ if and only if it holds for $G\vee K_1$.
\end{lemma}

\begin{proof}
	Let $G$ be a graph with $p$ vertices.
	Clearly, $e(G\vee K_1)=e(G)+p$ and from Lemma \ref{lm-1} (iii), it follows that
	$$\mu(G\vee K_1)=\{p+1, \mu_1(G)+1, \cdots, \mu_{p-1}(G)+1, 0\}.$$
	Thus, for $1\leq k\leq p$, we have (with a convention that $s_{0}(G)=0$)
	\begin{eqnarray}
		s_{k}(G\vee K_1)=p+k+s_{k-1}(G).\label{eq-4-1}
	\end{eqnarray}

    Suppose first that the full Brouwer's conjecture holds for $G$,
    that is, for $1\leq k\leq p$, we have (with a convention that $s_{0}(G)=0$ and $\binom{1}{2}=0$)
    \begin{eqnarray*}
    s_{k-1}(G)\leq e(G)+\binom{k}{2}, 
    \end{eqnarray*}
	with equality if and only if $G$ is a threshold graph having $p$ vertices and clique number $k$,
    or equivalently, $G\vee K_1$ is a threshold graph having $p+1$ vertices and clique number $k+1$.
    This, together with (\ref{eq-4-1}), yields that for $1\leq k\leq p$,
	\begin{eqnarray*}
		s_{k}(G\vee K_1)\leq p+k+e(G)+\binom{k}{2}=e(G\vee K_1)+\binom{k+1}{2},
	\end{eqnarray*}
	with equality if and only if $G\vee K_1$ is a threshold graph having $p+1$ vertices and clique number $k+1$.

	Conversely, suppose that the full Brouwer's conjecture holds for $G\vee K_1$,
    that is, for $1\leq k\leq p$,
    $$s_{k}(G\vee K_1)\leq e(G\vee K_1)+\binom{k+1}{2}=e(G)+p+k+\binom{k}{2},$$
    with equality if and only if $G\vee K_1$ is a threshold graph having $p+1$ vertices and clique number $k+1$,
    or equivalently, $G$ is a threshold graph having $p$ vertices and clique number $k$.
	This, as well as (\ref{eq-4-1}), yields that for $1\leq k\leq p$,
	\begin{eqnarray*}
		s_{k-1}(G)\leq e(G)+\binom{k}{2},
	\end{eqnarray*}
    with equality holding if and only if $G$ is a threshold graph having $p$ vertices and clique number $k$.
    This completes the proof of Lemma \ref{thm-4-1}.
\end{proof}

For a graph $G$ with $p$ vertices, it is easy to see that
$$\mu(G\cup K_1)=\{\mu_1(G),\mu_2(G),\cdots,\mu_{p-1}(G),0,0\},$$
which implies that
$$s_k(G\cup K_1)=s_k(G)\,\, \textrm{for}\,\, 1\leq k\leq p-1,\,\, \textrm{and}\,\, s_p(G\cup K_1)=s_{p-1}(G).$$
Also, $G$ is a threshold graph if and only if so is $G\cup K_1$, and $e(G\cup K_1)=e(G)$ and $\omega(G\cup K_1)=\omega(G)$.
Combining the above we can obtain the following.

\begin{lemma} \label{thm-4-2}
    The full Brouwer's conjecture holds for a graph $G$ if and only if it holds for $G\cup K_1$.
\end{lemma}

Note that the convention that $s_{0}(G)=0$ and $\binom{1}{2}=0$ yields that
$$s_0(K_1)=e(K_1)+\binom{1}{2}.$$
Now, by Lemmas \ref{thm-4-1} and \ref{thm-4-2}, as well as the constructive characterization of threshold graphs, we can deduce the next result
that was proven previously in the proof of Theorem \ref{thm-split} using a different method.

\begin{coro} \label{thm-4-3}
	With a convention that $s_{0}(G)=0$ and $\binom{1}{2}=0$,
    the full Brouwer's conjecture holds for threshold graphs.
\end{coro}

On the other hand, by refining some arguments that were used in \cite{Du} (where Fan's inequality is an indispensable tool),
we have also proven that the full Brouwer's conjecture is true for $c$-cyclic graphs with $c\in\{0,1,2\}$,
where Lemma \ref{lm-3-2-1} (not relied on Fan's inequality) plays a key role and might has some other potential applications.
In addition, we have proposed the Nordhaus-Gaddum version of the full Brouwer's conjecture (i.e., Conjecture \ref{conj-4})
and have presented some partial solutions to it, where 
Theorem \ref{thm-5} tells us that for sufficient large $n$,
Conjecture \ref{conj-4} holds for the ``sparse'' graphs such as planar graphs and $c$-cyclic graphs with small $c$ and for the ``most dense'' graphs,
while Theorem \ref{thm-6} tells us that for sufficient large $n$,
Conjecture \ref{conj-4} holds for nearly-regular graphs such as regular graphs and those graphs with maximum degree $\Delta<n/2$ or minimum degree $\delta>n/2$.

We finally would like to mention that there is a higher-dimensional generalization of Brouwer's conjecture; see \cite{Abebe} for details.
It would be natural to further consider the full version of it, based on the equivalent version of the full Brouwer's conjecture (see Conjecture \ref{conj-3}).

\section*{Acknowledgements}
This work was supported by the National Natural Science Foundation of China (No. 11861011)
and Natural Science Foundation of Guangxi Province (No. 2024GXNSFAA010516).

\section*{Declaration of Interest Statement}
We declare that we have no conflicts of interest.


\end{document}